\documentclass[a4paper, 11pt, oneside]{article}
\usepackage[utf8]{inputenc}
\usepackage{amsmath}
\usepackage{amssymb}
\usepackage{amsthm}
\usepackage{cite}
\usepackage{hyperref}
\usepackage{  dsfont  }

\newtheorem*{rep@theorem}{\rep@title}
\newcommand{\newreptheorem}[2]{%
\newenvironment{rep#1}[1]{%
 \def\rep@title{#2 \ref{##1}}%
 \begin{rep@theorem}}%
 {\end{rep@theorem}}}
\makeatother

\newtheorem{theorem}{Theorem}
\newreptheorem{theorem}{Theorem}
\newtheorem{lemma}{Lemma}
\newreptheorem{lemma}{Lemma}

\theoremstyle{definition}
\newtheorem{definition}{Definition}

\newcommand{\spann}[1]{\text{span}(#1)}

\title{A lower bound for the size of Kakeya sets with respect to hyperplanes in $\mathds{F}^n_q$}
\author{Beat Zurbuchen}
\date{\vspace{-4ex}}

\begin{document}
\maketitle
\begin{abstract}
We prove that a subset of $\mathds{F}_q^n$ that contains a hyperplane in any direction has size at least $q^{n}-O(q^2)$. 
\end{abstract}
\section{Introduction}
The classical Euclidean Kakeya problem is an important problem in harmonic analysis: let $E \subset \mathds{R}^n$ be a compact subset. We say $E$ is Kakeya if it contains a unit line segment in every direction. The problem claims that if $E$ is Kakeya then
\[
 \dim E = n
\]
where $\dim E$ denotes for either the Hausdorff or the Minkowski dimension (see \cite{Tao2}).  

To give a better understanding of the problem we explain the Minkowski dimension. Let $E \subset \mathds{R}^n$ and $\delta > 0$. Then we define 
$E_{\delta} = \{x \in \mathds{R}^n|\text{dist}(x, E) < \delta\}$ with $\text{dist}(x, E) = \inf\{\|x - a\|\hspace{0.125em} |a \in E\}$. 
One may think of $E_{\delta}$ as the union of all open balls with radius $r$ around every point in $E$. 

This allows us to define the Minkowski dimension for some $E \subset \mathds{R}^n$ as
\begin{equation} \label{eq:dim}
 \dim E := n - \lim_{\delta \rightarrow 0}\frac{\log \text{vol}(E_{\delta})}{\log \delta}.
\end{equation}
A good example to see how the definition works is to compute the dimension of a sphere $S \subset \mathds{R}^3$ of radius $r$. Then $S_\delta = \{x \in \mathds{R}^3|r - \delta < ||x|| < r + \delta\}$. This denotes for a set that contains the open ball with radius $r + \delta$ where 
the ball with radius $r - \delta$ has been cut out. Thus 
\begin{align*}
\text{vol}(S_{\delta}) &= \frac{4\pi}{3}((r + \delta)^3 - (r - \delta)^3) \\
&= \frac{4\pi\delta}{3}(6r^2 + 2\delta^2)
\end{align*}
which implies by (\ref{eq:dim}) that
\begin{align*}
 \dim S &= 3 - \lim_{\delta \rightarrow 0}\frac{\log\frac{4\pi}{3} + \log \delta + \log (6r^2 + \delta^2)}{\log \delta} \\
 &= 2.
\end{align*}
If $B \subset \mathds{R}^3$ is a ball with radius $r$ instead, then 
\[
\text{vol}(B_\delta) = \frac{4\pi}{3}(r + \delta)^3
\]
which means by (\ref{eq:dim}) that
\begin{align*}
 \dim B &= 3 - \lim_{\delta \rightarrow 0}\frac{\log\frac{4\pi}{3} + 3\log(r + \delta) }{\log \delta} \\
 &= 3. 
\end{align*}

It is however very hard to solve the Euclidean Kakeya problem which is why the problem is still open for $n > 2$. The case $n = 2$ was solved by Davies in
\cite{Davies71}.

In 1999, T. Wolff proposed a disrete analogue of the Kakeya problem. Let $\mathds{F}_q$ be the field with $q$ elements. We define a line to be the translate of a one-dimensional linear subspace. The direction of a given line $w'$ is the unique one-dimensional subspace $w$ such that $w'$ is a translate of $w$. A set $E \subset
\mathds{F}_q^n$ is Kakeya if $E$ contains a line in every direction, i. e. for all $v \in \mathds{F}_q^n\backslash \{0\}$ there is a $x_0 \in \mathds{F}^n_q$ such that $\{x_0 + tv | t \in \mathds{F}_q^n\} \subset E$. The problem then claims that for 
all $n$ there is a $c > 0$ (only depending on $n$) such that for all Kakeya sets $E \subset \mathds{F}_q^n$
\[
	|E| \geqslant cq^n.
\]
Z. Dvir was able to solve the problem with $c = \frac{1}{n!}$ using the so-called polynomial method.

One generalization of the Kakeya problem is the $k$-plane Furstenberg set problem in $\mathds{F}_q^n$ (see \cite{Ellb15}, Question 1.3). Let $S_k$ be the set of all subspaces $w \subset \mathds{F}_q^n$ such
that $\dim(w) = k$. We define a $k$-plane $w' \subset \mathds{F}_q^n$ to be a translated $k$-dimensional subspace. Given a $k$-plane $w'$, its direction is defined as the unique subspace $w \in \mathds{F}_q^n$ such that $w'$ is a translate of $w$. Fix some $c > 0$ and let $E \subset \mathds{F}_q^n$ be such that for every direction $w \in S_k$ there is at least one $k$-plane $w'$ with $|w'\cap E| \geqslant q^c$. The problem asks for a lower bound for $|E|$.

This problem implies the finite field Kakeya problem namely when $c = k = 1$. Note that $c = k$ means $w' \subset E$. We therefore consider this case as the generalization of the Kakeya problem. 
\begin{definition}[Kakeya set with respect to $k$-planes]
A set $E \subset \mathds{F}_q^n$ is Kakeya with respect to $k$-planes if for every $w \in S_k$ there is a $k$-plane $w' \subset \mathds{F}_q^n$ in direction $w$ such that $w' \subset E$. 
\end{definition}

We will only discuss the size of Kakeya sets w.r.t. to $(n-1)$-planes also known as hyperplanes and will therefore call a set Kakeya if it is Kakeya w.r.t. hyperplanes. From the result from Dvir one would expect that for every $n$ there is a constant $c > 0$ such that for every Kakeya set $E \subset \mathds{F}_q^n$ the inequality $|E| \geqslant cq^n$ from the original problem holds. In fact, we are able to show the following theorem.
\begin{theorem}
	Every set $E \subset \mathds{F}^n_q$ that is Kakeya fulfills
 \[
  |E| \geqslant \frac{q^{2n} - q^n}{q^n + q^2 - 2q}.
 \]
\end{theorem}

A known proof, independent of Dvirs method, solves the problem when $n = 2$ and gives 
\[
	|E| \geqslant \frac{q(q+1)}{2}
\]
as a bound for every set $E \subset \mathds{F}_q^2$ that is Kakeya. This proof is a discrete version of the previously mentioned proof in \cite{Davies71}. We were not able to trace it back in the literature, but to read an exposition of this proof see \cite{Kaloyan}.

We modify and generalize the proof so that it gives Theorem 1. What is surprising is that while for $n = 2$ the asymptotic bound is $|E| \geqslant \frac{q^2}{2} - O(q)$, for $n > 2$ it is $|E| \geqslant q^n - O(q^2)$ rather than the expected $|E| \geqslant cq^n - O(q^{n-1})$ with some $c \in (0, 1)$. 
\section{The proof}
Before being able to prove the theorem we need to prove two lemmas. Let $S = S_{n-1}$.
\begin{lemma}
 The set $S$ fulfills 
 \[
  |S| = \frac{q^n - 1}{q - 1}.
 \]

 \end{lemma}
\begin{proof}
 Let
 \[
 K = \{(k_1, \dots, k_{n-1}) \in  \mathds{F}^n_q\times \cdots \times \mathds{F}^n_q| \dim{\text{span}(k_1, \dots, k_{n-1})} = n - 1\}.
 \]
We may choose any element of $\mathds{F}^n_q \backslash \{0\}$ for
 $k_1$. If $k_2 \in \spann{k_1}$ then $\spann{k_1, k_2} = \spann{k_1}$. If however $k_2 \notin \spann{k_1}$ then $\spann{k_1} \subsetneq \spann{k_1, k_2}$
 and hence $\dim \spann{k_1, k_2} = 2$. Thus $k_2 \in \mathds{F}_q^n\backslash\spann{k_1}$. One can easily generalize this argument and conclude that $k_h \in \mathds{F}^n_q\backslash \spann{k_1, \dots, k_{h-1}}.$
 Therefore
 \[
  |K| = \prod_{h = 0}^{n -2}(q^n - q^h).
 \]

Define $\phi:K \mapsto S$ to be 
 $\phi(k_1, \dots, k_{n-1}) = \spann{k_1,\dots,  k_{n-1}}$. Thus
 \[
  |K| = \sum_{w \in S} |\text{fiber}_\phi(w)|.
 \]
 We count the number of $(k_1, \dots, k_{n-1}) \in \text{fiber}_{\phi}(w)$. We may choose any element from $w\backslash \{0\}$ for $k_1$.
 Using a similar argument as above we see that $k_h \in w\backslash \spann{k_1, \dots, k_{h-1}}$ which shows that
 \[
|\text{fiber}_{\phi}(w)| = \prod_{h = 0}^{n -2}(q^{n-1} - q^h).
\]

Therefore
\begin{align*}
  |S| &= \prod_{h = 0}^{n -2}\frac{q^{n} - q^h}{q^{n-1} - q^h} \\
  &= \prod_{h = 1}^{n-1}\frac{q^{h+1} - 1}{q^{h} - 1} \\
  &= \frac{q^n - 1}{q - 1}.
\end{align*} 
 \end{proof}
 \begin{lemma}
  Let $A$ and $B$ be finite sets and let $\phi:A \rightarrow B$ be a map. If $C = \{(a_1, a_2) \in A\times A| \phi(a_1) = \phi(a_2)\}$ then
  \begin{displaymath}
   |B| \geqslant \frac{|A|^2}{|C|}.
  \end{displaymath}
 \end{lemma}
\begin{proof}
 Let $b \in B$ and define $C_b = \{(a_1, a_2) \in C|\phi(a_1) = b\}$. Note that 
 \[
  |C| = \sum_{b \in B} |C_b|.
 \]
For a fixed $b$ one may choose any $a_1 \in \text{fiber}_\phi(b)$ and any $a_2 \in \text{fiber}_\phi(b)$. Thus $|C_b| = |\text{fiber}_\phi(b)|^2$. By applying the Cauchy-Schwarz
inequality we obtain
\[
  \Big(\sum_{b \in B}|\text{fiber}_\phi(b)| \cdot 1\Big)^2 \leqslant \sum_{b \in B}|\text{fiber}_\phi(b)|^2 \sum_{b \in B}1
\]
which is by the previous observations equivalent to the lemma. 
\end{proof}
\begin{proof}[Proof of Theorem 1]
For every $w \in S$ take a $v_w \in \mathds{F}_q^n$ such that $w' = v_w + w$ is contained in $E$. Let
\[ 
 I = \{(w, v) \in S\times E|v \in w'\}.
\] 
There are $|S|$ possibilities for $w$ while every $w$ contributes $q^{n-1}$ points since $|w'| = q^{n-1}$. Thus $|I| = |S|q^{n-1}$. 

Let $\gamma: I \rightarrow E$ be $\gamma(w, v) = v$ and define 
\[
 W = \{(w_1, w_2, v) \in S\times S\times E|v \in w_1'\cap w_2'\}.
\]
It is easy to see that there is a bijection between $W$ and $K = \{((w_1, v_1), (w_2, v_2)) \in I \times I|v_1 = v_2\}$ given by $(w_1, w_2, v) \mapsto ((w_1, v), (w_2, v))$. This implies by Lemma 2 that
\begin{equation} \label{eq:final}
 |E| \geqslant \frac{|I|^2}{|W|}.
\end{equation}
The last step before being able to obtain the final inequality is to compute $|W|$. To do so, we differ two cases:
\begin{enumerate}
 \item $w_1 = w_2$: This case contributes $|I|$ elements since for every $(w_1, v) \in I$ there is a $(w_1, w_1, v) \in W$ and vice versa. 
 \item $w_1 \neq w_2$: Note first that by assumption $w_1 \nparallel w_2$. By the dimension formula for subscpaces $w_1' \cap w_2'$ has cardinality $q^{n-2}$. There  are
 $|S|(|S| - 1)$ ordered pairs of not parallel hyperplanes. Hence this case contributes $|S|(|S| - 1)q^{n-2}$ elements to $W$.
\end{enumerate}
This leaves us with
\[
 |W| = |I| + |S|(|S| - 1)q^{n-2}.
\]
By using (\ref{eq:final}) and Lemma 1 we deduce that
\begin{align*}
 |E| &\geqslant \frac{|S|^2q^{2n-2}}{(|S|^2 - |S| + |S|q)q^{n-2}} \\
 &= \frac{q^{2n} - q^n}{q^n + q^2 - 2q}
\end{align*}
\end{proof}
Note how our final bound for $|E|$ has the form $|E| \geqslant q^n - O(q^2)$ for $n > 2$  and $|E| \geqslant \frac{q^2}{2} - O(q)$ for $n = 2$. Namely, when $n > 2$ the leading term in the denominator is $q^n$ while it is $2q^2$ when $n = 2$.

\section{Acknowledgements}
I want to thank SwissMAP for providing me with the opportunity to participate in this program. I also want to thank Kaloyan Slavov for his support throughout the whole project and his proposal to do this research project.
\bibliography{literature}
\bibliographystyle{alpha}
\end{document}